\documentclass[11pt,a4paper,reqno]{amsart}

\usepackage{indentfirst,enumerate,a4wide,fancyhdr,comment}
\usepackage{tikz}

\usetikzlibrary{hobby}

\usepackage{pst-node}
\usepackage[looser]{newpxtext}
\usepackage[smallerops]{newpxmath}
\usepackage[colorlinks=true, linktocpage=true, linkcolor=red, citecolor=red]{hyperref}

\newtheorem{Definition}{Definition}[section]
\newtheorem{Theorem}[Definition]{Theorem}

\newtheorem{Corollary}[Definition]{Corollary}
\newtheorem{Proposition}[Definition]{Proposition}
\newtheorem{Lemma}[Definition]{Lemma}

\numberwithin{equation}{section}
\numberwithin{figure}{section}

\title[A remark on rapid mixing for hyperbolic flows]{A remark on rapid mixing for hyperbolic flows}
\date{\today}

\author[DAOFEI ZHANG]{Daofei Zhang}
\address{School of Mathematics and Computing Science, Guilin University of Electronic Technology, Guilin, 541004, China}
\email{Daofei.Zhang@guet.edu.cn}

\begin{document}

\begin{abstract}
We establish an improved criterion for rapid mixing of hyperbolic flows by weakening the requirement on the temporal distance function from positive box dimension to the existence of two values whose ratio is Diophantine. We also demonstrate the applicability of our results through explicit examples where the previous dimension condition were either too restrictive or computationally infeasible to verify. 
\end{abstract}

\maketitle

\section{Introduction}\label{sec 1}

\subsection{Statement of main results}\label{subsec 1.1}
Estimating the mixing rates of smooth flows with some hyperbolicity is a highly challenging problem. Even in the classical case of uniformly hyperbolic flows, the problem has not been fully understood either. There are examples for which the mixing may be arbitrarily slow \cite{Pol85}. But in contrast to this there are positive results show rapid mixing under a Diophantine condition on closed orbits or a dimension condition on the temporal distance function. In this paper, as already stated in the abstract, we improve the above criterion on the temporal distance function by reducing the positive box dimension condition to the existence of two Diophantine-related values. We will also give examples illustrate the improvement is indeed useful and applicable by showing that the previous dimension condition cannot be verified

We begin by recalling a classic result of Dolgopyat from \cite{Dol98b}. Let \( M \) be a compact smooth Riemannian manifold, and let \( g_t: \Lambda \to \Lambda \subset M \) be a \( C^\infty \) hyperbolic flow. Denoted by \(W^s\) and \(W^u\) the stable and unstable manifolds respectively. Given \(\varepsilon > 0\) small, let $W^s_{\varepsilon}$ and $W^u_{\varepsilon}$ be the local stable and unstable manifolds of size \(\varepsilon\) respectively. There exists \(\varepsilon_0 > 0\) such that, for any \(z_0 \in \Lambda\) and any two points \(z_1, z_2 \in \Lambda\) with \(z_1 \in W^u_{\varepsilon_0}(z_0)\) and \(z_2 \in W^s_{\varepsilon_0}(z_0)\), the intersection \(W^s_{\varepsilon_0}(z_1) \cap \bigcup_{|t| \le \varepsilon_0} g_t W^u_{\varepsilon_0}(z_2)\) consists of a single point which belongs to \(\Lambda\), and we denote it by \([z_1, z_2]\), often referred to as the local product of \(z_1\) and \(z_2\) \cite{Bow75}. Let us set \(z_4 := [z_1, z_2]\). In other words, there exist unique \(z_3 \in W^u_{\varepsilon_0}(z_2)\) and \(\mathcal{T}(z_1, z_2) \in [-\varepsilon_0, \varepsilon_0]\) such that \(g_{\mathcal{T}(z_1, z_2)}(z_3) = z_4\), which we also illustrated in Figure \ref{Figure 1}.  As a function of \(z_1\) and \(z_2\), we refer to \(\mathcal{T}\) as the temporal distance function, which is defined on \(W^u_{\varepsilon_0}(z_0) \cap \Lambda \times W^s_{\varepsilon_0}(z_0) \cap \Lambda\). Denote by \(\text{range}(\mathcal{T})\) the range of \(\mathcal{T}\). Let \( \mu_\Phi \) denote the Gibbs measure associated with a Hölder continuous potential \( \Phi \) on \( \Lambda \). In \cite{Dol98b}, Dolgopyat shown the following criterion on rapid mixing for hyperbolic flows.

\begin{Theorem}[Dolgopyat \cite{Dol98b}]\label{Theorem 1.1}
If  \(\emph{range}(\mathcal{T})\) has positive lower box dimension, then the hyperbolic flow \(g_t\) is rapidly mixing with respect to \(\mu_{\Phi}\). Specifically, the quantity
$$\int E\circ g_{t}. Fd\mu_{\Phi}-\int Ed\mu_{\Phi}\int Fd\mu_{\Phi}$$
decays to zero faster than any polynomial rate as $t\to\infty$ for any smooth functions $E, F$ on $M$.
\end{Theorem}

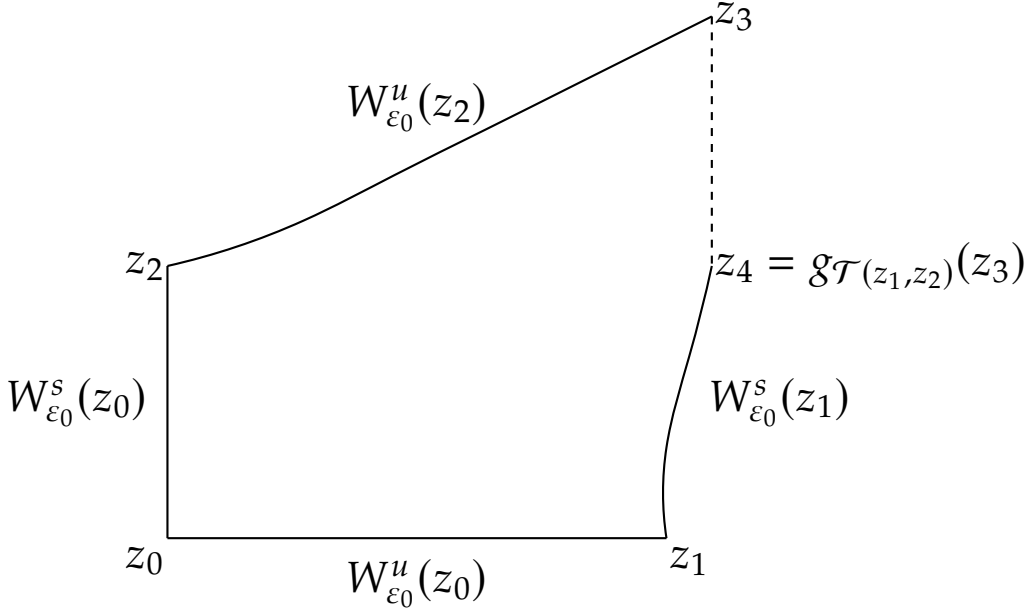
\begin{figure}[h!]\label{Figure 1}
	\centerline{
		\begin{tikzpicture}[thick,scale=0.60, every node/.style={scale=1.5}]
			\draw[-, black] (0,0)-- (0,6);
			\draw[-, black] (0,0)-- (11,0);
			\draw plot [hobby] coordinates {(0,6) (3,7) (5,8) (7,9) (9,10) (11,11) (12,11.5)};
			\draw plot [hobby] coordinates {(11,0) (11,2) (11.5,4) (12,6)};
			\draw[-, dashed] (12,11.5)-- (12,6);
			\node at (-0.5,-0.5) {$z_0$};
			\node at (11.5,-0.5) {$z_1$};
			\node at (5.5,-1) {$W^{u}_{\varepsilon_{0}}(z_{0})$};
			\node at (-2,3) {$W^{s}_{\varepsilon_{0}}(z_{0})$};
			\node at (13.5,3) {$W^{s}_{\varepsilon_{0}}(z_{1})$};
			\node at (5.5,9.5) {$W^{u}_{\varepsilon_{0}}(z_{2})$};
			\node at (-0.5,6) {$z_2$};
			\node at (15.5,6) {$z_4=g_{\mathcal{T}(z_{1},z_{2})}(z_{3})$};
			\node at (12.5,11.5) {$z_3$};
		\end{tikzpicture}
	}
	\caption{The temporal distance $\mathcal{T}(z_{1},z_{2})$ of $z_{1}$ and $z_{2}$ where $z_{4}=[z_{1},z_{2}]$}
\end{figure}

The positive box dimension condition in Theorem \ref{Theorem 1.1} could be satisfied if \(g_t\) is a jointly non-integrable Anosov flow, since in this case \(\text{range}(\mathcal{T})\) contains a non-empty interval. For a hyperbolic flow with \(\Lambda\) has positive box dimension, if the flow preserves a contact  structure, it shown in \cite{Mel09} that \(\text{range}(\mathcal{T})\) also has positive box dimension. However, for general hyperbolic flows, it is not clear how to verify the above dimension condition, and it may  occurs that \(\text{range}(\mathcal{T})\) at most be a countable set (e.g., the example in Subsection \ref{subsec 1.2}). To overcome this problem, we improve the above criterion by reducing the positive box dimension condition to the existence of two Diophantine-related values. Recalling a Diophantine number \( \alpha \in \mathbb{R} \) means that \( |q\alpha - p| \geq C|q|^{-\gamma} \) for some constants \( C > 0 \) and \( \gamma > 0 \), for any \( p \in \mathbb{Z} \) and any nonzero \( q \in \mathbb{Z} \).

\begin{Theorem}\label{Theorem 1.2}
If there exist $\alpha$ and $\beta\in \emph{range}(\mathcal{T})$ such that $\frac{\alpha}{\beta}$ is a Diophantine number, then $g_{t}$ is rapidly mixing with respect to $\mu_{\Phi}$, namely,  for any $n\in\mathbb{N}^{+}$ there exist $C>0$ and $k\in\mathbb{N}^{+}$ such that 
$$\bigg|\int E\circ g_t . F d \mu_{\Phi}  - \int E d \mu_{\Phi} \int Fd \mu_{\Phi}\bigg|\le C||E||_{C^{k}}||F||_{C^{k}}t^{-n},$$
for any $E, F\in C^{k}(M)$ and any $t>0$.
\end{Theorem}

The above Diophantine condition is strictly weaker than the dimension of range$(\mathcal{T})$ is positive. This follows from two observations. One is the set of non-Diophantine numbers is zero-dimensional \cite{Bug05}, and thus any set with positive dimension must contain points $\alpha$ and $\beta$ with $\frac{\alpha}{\beta}$ is a Diophantine number. Another is that in Subsection \ref{subsec 1.2} we construct explicit examples illustrate the corresponding range$(\mathcal{T})$ is zero-dimensional but the Diophantine condition in Theorem \ref{Theorem 1.2} is satisfied.

Under the same assumption as in Theorem \ref{Theorem 1.2}, following the argument in \cite{Pol01}, we can also obtain a polynomial error term in the Prime Orbit Theorem. Let \(\tau\) represents a closed orbit of \(g_t\), and denote its period as \(\ell_\tau\). For any \(T > 0\), let \(\pi(T)\) be the collection of prime closed orbits \(\tau\) with \(\ell_\tau \leq T\).

\begin{Theorem}\label{Theorem 1.3}
If there exist $\alpha$ and $\beta\in \emph{range}(\mathcal{T})$ such that $\frac{\alpha}{\beta}$ is a Diophantine number, then there exist $C>0$ and $\delta>0$ such that for any $T>0$,
$$
\big|\#\pi(T)-\emph{li}(e^{hT})\big|\le C\frac{e^{hT}}{{T^{1+\delta}}}, 
$$
where $h$ denotes the topological entropy of $g_{t}$ and $\emph{li}(T)=\int^{T}_{2}\frac{1}{\log u}du$.
\end{Theorem}

In addition to the rapid mixing result in Theorem~\ref{Theorem 1.2} and the counting result in Theorem~\ref{Theorem 1.3}, we can also apply the results in \cite{Mel02} to obtain a class of limit theorems. These include the central limit theorem, the law of the iterated logarithm, and the almost sure invariance principle for the time-one map of the hyperbolic flow.

\subsection{Examples}\label{subsec 1.2}

We now present two classes of hyperbolic flows that satisfy the Diophantine condition in Theorem~\ref{Theorem 1.2} but fail to meet previous requirements for rapid mixing. These examples demonstrate the broader applicability of our results.

Our first class consists of suspension flows over hyperbolic diffeomorphisms. To express more precise, we now introduce some notation. Let \( M \) be a compact smooth Riemannian manifold, and let \( f: \Lambda \to \Lambda \subset M \) be a \( C^\infty \) hyperbolic diffeomorphism. We assume that $\Lambda\subset M$ is a cantor set. Consider a $C^{\infty}$ positive function $r:M\to \mathbb{R}$. We define the suspension space $\Lambda_{r}=\{(x,u):0\le u\le r(x)\}\sim$ where $(x,r(x))\sim (f(x),0)$, and the suspension flow $\phi_{t}$ of $f$ under $r$ by
$$\phi_{t}:\Lambda_{r}\to\Lambda_{r},\quad \phi_{t}(x,u)=(x,u+t),$$
with respect to $\sim$ on $\Lambda_{r}$. It is standard that $\phi_{t}:\Lambda_{r}\to\Lambda_{r}$ is a hyperbolic flow. Let $W^{s}_{f}$ and $W^{u}_{f}$ be the stable and unstable manifolds of $f$ respectively, and also denote by $W^{s}_{g_{t}}$ and $W^{u}_{g_{t}}$ the stable and unstable manifolds of $g_{t}$ respectively. They have the following relationship.

\begin{Lemma}\label{Lemma 1.4}
For any $(x,u)\in\Lambda_{r}$, we have $W^{s}_{g_{t}}(x,u)=\{(y,u+\Delta^{s}(x,y)): y\in W^{s}_{f}(x)\}$ and $W^{u}_{g_{t}}(x,u)=\{(y,u+\Delta^{u}(x,y)): y\in W^{u}_{f}(x)\}$, where 
$$\Delta^{s}(x,y)=\sum_{n=0}^{\infty}r(f^{n}(y))-r(f^{n}(x))\quad\text{and}\quad\Delta^{u}(x,y)=\sum_{n=1}^{\infty}r(f^{-n}(x))-r(f^{-n}(y)).$$
\end{Lemma}
\begin{proof}
Assume $(y,u^{\prime})\in W^{s}_{g_{t}}(x,u)$, we then have 
$$
\begin{aligned}
0=&\lim_{t\to+\infty}d(\phi_{t}(x,u),\phi_{t}(y,u^{\prime}))\\
=&\lim_{t\to+\infty}d((x,u+t),(y,u^{\prime}+t))\\
=&\lim_{t\to+\infty}d((f^{n}(x),u+t-r_{n}(x)),(f^{n}(y),u^{\prime}+t-r_{n}(y))),
\end{aligned}
$$
where $r_{n}=\sum_{i=0}^{n-1}r\circ f^{i}$. From which we deduce that $\lim_{n\to\infty}d(f^{n}(x),f^{n}(y))=0$ and $u^{\prime}=u+\sum_{n=0}^{\infty}r(f^{n}(y))-r(f^{n}(x))$. Thus the case  of $W^{s}_{g_{t}}(x,u)$ is proved. The proof of the second case of $W^{u}_{g_{t}}(x,u)$ is similar.
\end{proof}

In particular, given $(x_{0},u_{0})\in\Lambda_{r}$, $(x_{1},u_{1})\in W^{s}_{g_{t}}(x_{0},u_{0})\cap\Lambda_{r}$ and $(x_{2},u_{2})\in W^{u}_{g_{t}}(x_{0},u_{0})\cap\Lambda_{r}$ with $x_{1},x_{2}\in B(x_{0},\delta)$, by Lemma \ref{Lemma 1.4}, we have
$$\mathcal{T}((x_{1},u_{1}),(x_{2},u_{2}))=\Delta^{s}(x_{0},x_{1})+\Delta^{u}(x_{1},x_{3})+\Delta^{s}(x_{3},x_{2})+\Delta^{u}(x_{2},x_{0}),$$
where $x_{3}=W^{u}_{f}(x_{1})\cap W^{s}_{f}(x_{2})\cap B(x_{0},\delta)$. Thus, we obtain that 
$$
\text{range}(\mathcal{T})=\{\mathcal{H}(x_{1},x_{2}): x_{1}\in W^{s}_{f}(x_{0})\cap B(x_{0},\delta),\ x_{2}\in W^{u}_{f}(x_{0})\cap B(x_{0},\delta)\},
$$
where $\mathcal{H}(x_{1},x_{2})=\Delta^{s}(x_{0},x_{1})+\Delta^{u}(x_{1},x_{3})+\Delta^{s}(x_{3},x_{2})+\Delta^{u}(x_{2},x_{0})$.

\begin{Theorem}\label{Theorem 1.5}
For any $\alpha,\beta\in\mathbb{R}$, we can choose our roof function $r:\Lambda\to\mathbb{R}$ such that $\emph{range}(\mathcal{T})=\{0,\pm\alpha,\pm\beta\}$.
\end{Theorem}
\begin{proof}
Following the preceding analysis, for any $x_{1}\in W^{s}_{f}(x_{0})\cap B(x_{0},\delta)\cap\Lambda_{r}$ and $x_{2}\in W^{u}_{f}(x_{0})\cap B(x_{0},\delta)\cap\Lambda_{r}$, we can write
$$
\begin{aligned}
\mathcal{H}(x_{1},x_{2})=&\Delta^{s}(x_{0},x_{1})+\Delta^{u}(x_{1},x_{3})+\Delta^{s}(x_{3},x_{2})+\Delta^{u}(x_{2},x_{0})\\
=&\Delta^{s}(x_{0},x_{1})+\Delta^{u}(x_{1},x_{3})+\Delta^{u}(x_{2},x_{0})+\sum_{n=1}r(f^{n}(x_{2}))-r(f^{n}(x_{3}))+r(x_{2})-r(x_{3})\\
=&\mathcal{G}(x_{1},x_{2})-r(x_{3}),
\end{aligned}
$$
where $\mathcal{G}(x_{1},x_{2})=\Delta^{s}(x_{0},x_{1})+\Delta^{u}(x_{1},x_{3})+\Delta^{u}(x_{2},x_{0})+\sum_{n=1}r(f^{n}(x_{2}))-r(f^{n}(x_{3}))+r(x_{2})$ is independent of the choice of $r(x_{3})$. The key observation is that for any pair $(x_{1},x_{2})$, since $\Lambda$ is a cantor set, we can freely adjust the value $r(x_{3})$ to ensure $\mathcal{H}(x_{1},x_{2})$ attains any desired value in $\{0,\pm\alpha,\pm\beta\}$. This completes the proof.
\end{proof}

Now, by the above theorem, we could choose $\alpha$ and $\beta$ such that $\frac{\alpha}{\beta}$ is a Diophantine number. Meanwhile, it is obvious that the dimension of $\text{range}(\mathcal{T})$ is zero.

Our second class consists of carefully constructed perturbations of contact hyperbolic flows. Consider a hyperbolic flow $g_t: \Lambda \to \Lambda$ on a $(2n+1)$-dimensional manifold $M$ preserving a contact form $\varrho$ (i.e., $\varrho \wedge (d\varrho)^n$ is nowhere vanishing). It is proved in \cite{Mel09} that for any contact hyperbolic flow, the dimension of $\text{range}(\mathcal{T})$ is positive. This fundamental result allows us to select $\alpha, \beta \in \text{range}(\mathcal{T})$ with $\alpha/\beta$ Diophantine. By definition, the values $\alpha$ only depends on four points in $\Lambda$. Thus, we could perturb the flow $g_{t}$ in a certain way that the temporal distance with respect to the perturbed flow $f_{t}:\Lambda\to\Lambda$ associated to these four points unchanged. For instance, the perturbed flow $f_{t}$ unchanged in a neighborhood of these four points. The same argument applied to the values $\beta$. Consequently, we have $\alpha$ and $ \beta$ still belong to $\text{range}(\mathcal{T})$ of the perturbed flow $f_{t}$, and thus $f_{t}$ satisfies the Diophantine condition in Theorem \ref{Theorem 1.2}. However, it doesn't have to be a contact flow and the dimension of $\text{range}(\mathcal{T})$ doesn't have to be positive since contact  structures are unstable. 

\subsection{Outline of the paper}\label{subsec 1.3}

The proofs of Theorems \ref{Theorem 1.2} and \ref{Theorem 1.3} follow the classical approach of Markov partitions and symbolic dynamics. This paper is organized as follows. In Section \ref{sec 2}, we recall basic definitions and notations of hyperbolic flows and symbolic dynamics. In Section \ref{sec 3}, via a Markov partition, we will deduce Theorems \ref{Theorem 1.2} and \ref{Theorem 1.3} to a Dolgopyat type estimate (Proposition \ref{Dolgopyat type estimate}) for the associated transfer operators. The argument is now classic and originated from \cite{Pol85}. In Section \ref{sec 4}, we prove Proposition \ref{Dolgopyat type estimate}, and the main ingredient in its proof is a cancellation estimate stated in Lemma \ref{Lemma 4.4}, whose proof will be given in Section \ref{sec 5}. The proof of Lemma \ref{Lemma 4.4} will use the conclusion obtained in Lemma \ref{Lemma 3.5}, which follows from the Diophantine assumption in Theorem \ref{Theorem 1.2}. 

\bigskip

\noindent\textbf{Acknowledgements}. The author expresses sincere gratitude to Ian Melbourne for his insightful comments and corrections, as well as to Paulo Varandas for his valuable suggestions and comments.

\section{Hyperbolic flows and symbolic dynamics}\label{sec 2}

This section establishes the fundamental concepts and notation for our analysis. We begin by recalling the definition of hyperbolic flows and their key dynamical objects. Following this, we review the construction of suspensions of subshifts of finite type, which serve as symbolic models for hyperbolic flows. 

\subsection{Hyperbolic flows}\label{subsec 2.1}

Let $M$ be a smooth compact Riemannian manifold, and let $TM$ be its tangent space. Let $g_{t}:M\to M$ be a $C^{1}$ flow, and let $\Lambda\subset M$ be a compact $g_{t}$-invariant set. We assume that $\Lambda$ consists of more than a single closed orbit, and closed orbits of $g_{t}$ in $\Lambda$ are dense. In simple terms, the flow $g_{t}:\Lambda\to \Lambda$ is called a hyperbolic flow if we can express $T_\Lambda M = E^c \oplus E^s \oplus E^u$, where $E^c$ is the line bundle of the flow, and $E^s$, $E^u$ are subspaces that remain invariant under the differential $Dg_t$ with $E^s$ contracting and $E^u$ expanding \cite{Bow75}, namely, there exist $C>0$ and $\delta>0$ such that $||Dg_{t}v||\le Ce^{-\delta t}||v||$ for all $v\in E^{s}$ and $t\ge0$;  $||Dg_{-t}v||\le Ce^{-\delta t}||v||$ for all $v\in E^{u}$ and $t\ge0$. Given a Hölder potential $\Phi$ on $\Lambda$, which is a Hölder continuous real-valued function, we can associate with it a $g_t$-invariant probability measure $\mu_{\Phi}$ called the Gibbs measure of $\Phi$ \cite{Bow75}. We are interested in the asymptotic behaviour of the following quantity.

\begin{Definition}\label{Def 2.1}
For $E,F : \Lambda \to \mathbb C$, their correlation function with respect to $\mu_{\Phi}$ is defined by 
$$
\rho_{E,F}(t): = \int_{\Lambda} E\circ g_t. Fd \mu_{\Phi}  - \int_{\Lambda} E d\mu_{\Phi} \int_{\Lambda} F d\mu_{\Phi}, \hbox{ for } t \in \mathbb R.
$$
\end{Definition}

The flow $g_{t}:  \Lambda \to\Lambda$ is said to be mixing with respect to  $\mu_{\Phi}$ if  $\rho_{E,F}(t) \to 0$ as $t \to +\infty$ for any $E$ and $F\in L^{2}(\mu_{\Phi})$. One of the most significant aspects of the study of the hyperbolic flow $g_{t}$ is the mixing rate (i.e., the rate of decay of $\rho_{E,F}(t)$) with respect to the Gibbs measure $\mu_{\Phi}$. It is important to recognize that $g_{t}$ is not always mixing with respect to $\mu_{\Phi}$. For example, $g_{t}$ is a constant suspension of a hyperbolic diffeomorphism. There are also examples where mixing occurs at an arbitrarily slow rate \cite{Pol85}. In this paper, we focus on a decay rate of $\rho_{E,F}(t)$ of the the following form.

\begin{Definition}\label{Def 2.2}
We say that the flow $g_{t}:\Lambda\to\Lambda$ is rapidly mixing with respect to $\mu_{\Phi}$ if for any $n\in\mathbb{N}^{+}$ there exist $C>0$ and $k\in\mathbb{N}^{+}$ such that for any $E, F\in C^{k}(M)$ and any $t>0$, we have $|\rho_{E,F}(t)| \leq C||E||_{C^{k}}||F||_{C^{k}}t^{-n}$.
\end{Definition}

From the definition, rapid mixing describes that the correlation functions of smooth functions decay faster than any polynomial and is thus also called superpolynomial mixing. It's proven in \cite{Dol98b} that if there exists \(z_0 \in \Lambda\) such that the range of the temporal distance function has positive lower box dimension, then the hyperbolic flow \(g_t\) is rapidly mixing with respect to \(\mu_{\Phi}\). In this paper, we improve this criterion by weakening the requirement on the temporal distance function from positive box dimension to the existence of two values whose ratio is Diophantine. Thus, we end this subsection by recalling the temporal distance function.

To begin, let's recall the stable and unstable manifolds. The subbundles \(E^s\) and \(E^u\) are always integrable. The integral manifolds of \(E^s\) and \(E^u\) are termed the stable and unstable manifolds, denoted by \(W^s\) and \(W^u\), respectively. They are \(g_t\)-invariant foliations. Given \(\varepsilon > 0\) small, the local stable manifold of \(x\) of size \(\varepsilon\) is defined by \(W^s_{\varepsilon}(x) = \{ y \in W^s(x) : \sup_{t \ge 0} d(g_t(y), g_t(x)) \le \varepsilon \}\). Similarly, the local unstable manifold of \(x\) of size \(\varepsilon\) is defined by \(W^u_{\varepsilon}(x) = \{ y \in W^u(x) : \sup_{t \ge 0} d(g_{-t}(y), g_{-t}(x)) \le \varepsilon \}\). When we don't need to emphasize the size, we will write them as \(W^s_{\text{loc}}\) and \(W^u_{\text{loc}}\). There exists \(\varepsilon_0 > 0\) such that, for any \(z_0 \in \Lambda\) and any two points \(z_1, z_2 \in \Lambda\) with \(z_1 \in W^u_{\varepsilon_0}(z_0)\) and \(z_2 \in W^s_{\varepsilon_0}(z_0)\), the intersection \(W^s_{\varepsilon_0}(z_1) \cap \bigcup_{|t| \le \varepsilon_0} g_t W^u_{\varepsilon_0}(z_2)\) consists of a single point which belongs to \(\Lambda\), and we denote it by \([z_1, z_2]\), often referred to as the local product of \(z_1\) and \(z_2\) \cite{Bow75}. Let's set \(z_4 := [z_1, z_2]\). In other words, there exist unique \(z_3 \in W^u_{\varepsilon_0}(z_2)\) and \(\mathcal{T}(z_1, z_2) \in [-\varepsilon_0, \varepsilon_0]\) such that \(g_{\mathcal{T}(z_1, z_2)}(z_3) = z_4\). The unique value \(\mathcal{T}(z_1, z_2)\) is called the temporal distance of \(z_1\) and \(z_2\). As a function of \(z_1\) and \(z_2\), we refer to \(\mathcal{T}\) as the temporal distance function, which is defined on \(W^u_{\varepsilon_0}(z_0) \cap \Lambda \times W^s_{\varepsilon_0}(z_0) \cap \Lambda\).

\subsection{Symbolic dynamics}\label{subsec 2.2}

The hyperbolic flow $g_{t}:\Lambda\to\Lambda$ can be modeled by a suspension flow of a two-sided subshift of finite type \cite{Bow73}. Below, we will briefly introduce these objects with \cite{Par90} serving as a helpful reference. 

Let $A$ be an $N\times N$ matrix of zeros and ones. We assume $A$ is aperiodic, meaning that some power of $A$ is a positive matrix. Let $X := X_{A}$ be the two-sided symbolic space associated with $A$, defined as $X = \{x=(x_{i})_{i=-\infty}^{\infty}\in\{1,\cdots,N\}^{\mathbb{Z}}: A(x_{i}, x_{i+1}) = 1\}$, and let $\sigma:X\to X$ be the two-sided subshift, defined as $(\sigma(x))_{i} = x_{i+1}$. Given $\lambda \in (0, 1)$, we can define a metric $d_{\lambda}$ on $X$ by $d_{\lambda}(x, y) := \lambda^{N(x,y)}$, where $N(x,y) = \min\{|i| : x_{i} \neq y_{i}\}$. For each $n\in\mathbb{N}^{+}$, the $n$-cylinders are sets of the form $[x_{-n+1}\cdots x_{0}\cdots x_{n-1}]_{n} = \{y \in X : y_{i} = x_{i}, |i| \le n-1\}$. Denote by $F_{\lambda}(X)$ the Banach space of all Lipschitz continuous complex-valued functions on $(X,d_{\lambda})$ with respect to the Lipschitz norm $||\cdot||_{\text{Lip}}:=|\cdot|_{\text{Lip}} + |\cdot|_{\infty}$, where $|\cdot|_{\text{Lip}}$ is the Lipschitz semi-norm and $|\cdot|_{\infty}$ is the supremum norm. Let \(F_{\lambda}(X,\mathbb{R})\) be the set of real-valued functions in \(F_{\lambda}(X,)\). All of the above objects can be similarly defined for a one-sided subshift of finite type \(\sigma:X^{+}\to X^{+}\), where \(X^{+}=\{x=(x_{i})_{i\ge0}\in\{1,\ldots,N\}^{\mathbb{N}} : A(x_{i}, x_{i+1})=1\}\), using the same notations but replacing \(X\) with \(X^{+}\).

Given a roof function $r\in F_{\lambda}(X,\mathbb{R}^{+})$, we define the suspension space as $X_{r}:=\{(x,u):0\le u\le r(x)\}/\sim$, where $(x,r(x))\sim(\sigma(x),0)$. Let $\phi_{t}:X_{r}\to X_{r}$
be the suspension flow, namely,
$$
\phi_{t}:X_{r}\to X_{r},\quad \phi_{t}(x,u)=(x, u+t),\quad t\in\mathbb{R},
$$
with respect to $\sim$ on $X_{r}$. For each $n\in\mathbb{Z}$ and any $x\in X$, we define 
$$W^{s}_{n}(x):=\{y\in X:y_{i}=x_{i},\ \forall i\ge n\}\quad\text{and}\quad W^{u}_{n}(x):=\{y\in X:y_{i}=x_{i},\ \forall i\le-n\}.$$
Let $W^{s}(x)=\bigcup_{n\in\mathbb{Z}}W^{s}_{n}(x)$ and $W^{u}(x)=\bigcup_{n\in\mathbb{Z}}W^{u}_{n}(x)$. Then $W^{s}(x)=\{y\in X:d_{\lambda}(\sigma^{n}(x),\sigma^{n}(y))\to0\text{ as }n\to\infty\}$ and  $W^{u}(x)=\{y\in X:d_{\lambda}(\sigma^{-n}(x),\sigma^{-n}(y))\to0\text{ as }n\to\infty\}$. For each $n\in\mathbb{N}^{+}$, let $r_{n}=\sum_{i=0}^{n-1}r\circ\sigma^{i}$.

\begin{Definition}\label{Def 2.3}
	For any $x,y\in X$ with $y\in W^{s}(x)$, we define $\Delta^{s}(x,y):=\lim_{n\to\infty} r_{n}(y)-r_{n}(x)$. Similarly, for any $x,y\in X$ with $y\in W^{u}(x)$, we define $\Delta^{u}(x,y):=\lim_{n\to\infty} r_{n}(\sigma^{-n}(x))-r_{n}(\sigma^{-n}(y))$. 
\end{Definition}

One should note that $\Delta^{s}(x,y)=-\Delta^{s}(y,x)$ as well as $\Delta^{u}(x,y)=-\Delta^{u}(y,x)$. We need to estimate the convergence rate of the limit $\Delta^{u}$.

\begin{Lemma}\label{Lemma 2.4}
	There exists $C_{1}>0$ such that for any $k\in\mathbb{N}^{+}$, any $n\in\mathbb{N}$ and any $x,y\in X$ with $y\in W^{u}_{n}(x)$,
	$$|\Delta^{u}(x,y)-[r_{k}(\sigma^{-k}(x))-r_{k}(\sigma^{-k}(y))]|\le C_{1}\lambda^{k-n}.$$
\end{Lemma}
\begin{proof}
	We compute,
	$$
	\begin{aligned}
		&|\Delta^{u}(x,y)-[r_{k}(\sigma^{-k}(x))-r_{k}(\sigma^{-k}(y))]|\\
		=&\bigg|\sum_{i=1}^{\infty}[r(\sigma^{-i}(x))-r(\sigma^{-i}(y))]-\sum_{i=1}^{k}[r(\sigma^{-i}(x))-r(\sigma^{-i}(y))]\bigg|\\
		=&\bigg|\sum_{i=k+1}^{\infty}[r(\sigma^{-i}(x))-r(\sigma^{-i}(y))]\bigg|\\
		\le&|r|_{Lip}\lambda^{k-n}\dfrac{1}{1-\lambda}.
	\end{aligned}
	$$
	Therefore, we can set $C_{1}=|r|_{Lip}\frac{1}{1-\lambda}$ to complete the proof. 
\end{proof}

\section{A Dolgopyat type estimate}\label{sec 3}

In this section, we prove Theorems~\ref{Theorem 1.2} and~\ref{Theorem 1.3} under the assumption of Proposition~\ref{Dolgopyat type estimate}, which provides a crucial Dolgopyat type estimate for the associated transfer operators. In Subsection~\ref{subsec 3.1}, we recall Bowen's classical result establishing a semi-conjugacy between the hyperbolic flow and a suspension flow over a two-sided subshift of finite type (introduced in Subsection~\ref{subsec 2.2}). Subsection~\ref{subsec 3.2} demonstrates how the Diophantine condition in Theorem~\ref{Theorem 1.2} induces a corresponding Diophantine property for the suspension flow (Lemma~\ref{Lemma 3.5}). Finally, under this Diophantine property, we state a Dolgopyat type estimate (Proposition~\ref{Dolgopyat type estimate}), whose proof will be developed in Sections~\ref{sec 4} and~\ref{sec 5}. The proofs of Theorems~\ref{Theorem 1.2} and~\ref{Theorem 1.3} then follow through standard arguments presented at the end of this section.

\subsection{Markov partitions}\label{subsec 3.1}

A Markov section of a hyperbolic flow allows us to build a bridge between the suspension flow introduced in the previous subsection 
and the hyperbolic flow. Fix a hyperbolic flow $g_{t}:\Lambda\to\Lambda$. For a sufficiently small $\delta>0$, recall the local product $[x,y]$ of two points $x,y\in\Lambda$ with $d(x,y)<\delta$. A subset $B\subset W^{s}_{\delta}(z)\cap\Lambda$ is called proper if it is closed and $\overline{B^{o}}=B$ with respect to the induced topology on $W^{s}_{\delta}(z)\cap\Lambda$. A proper subset in $W^{u}_{\delta}(z)\cap\Lambda$ is defined similarly.

\begin{Definition}
	A \emph{parallelogram} $R$ is a set of the form $R=[U,S]$ where $U\subset W^{u}_{\delta}(z)\cap\Lambda$ and $S\subset W^{s}_{\delta}(z)\cap\Lambda$ are proper subsets.
\end{Definition}

By definition, a parallelogram $R$ is a local cross section of the flow $g_{t}$. A parallelogram $R$ has disintegration with leaves $[x,S]$, namely $R=\cup_{x\in U}[x,S]$. Obviously, $[x,S]\subset W^{s}_{\delta}(x)\cap\Lambda$. Thus, we can  denote $[x,S]$ as $W^{s}(x,R)$. We denote by $W^{u}(x,R) = [U,x]$ where $x\in S$.

Let $\{R_{i}\}_{i=1}^{N}$ be finitely many parallelograms with $R_{i}\cap R_{j}=\emptyset$ for $i\not=j$. Fix a small $\varepsilon>0$, we assume $\cup_{i}\cup_{|t|\le\varepsilon}g_{t}R_{i}=\Lambda$. Let $R=\cup_{i}R_{i}$ and $P:R\to R$ correspond to the first return map of $g_{t}$ on $R$.

\begin{Definition}
	$\{R_{i}\}_{i=1}^{N}$ is called a Markov section of $g_{t}$ if whenever $x\in R_{i}^{o}\cap P^{-1}R_{j}^{o}\not=\emptyset$, 
	$$
	PW^{s}(x,R_{i})^{o}\subset W^{s}(P(x),R_{j})^{o}\quad\text{and}\quad P^{-1}W^{u}(P(x),R_{j})^{o}\subset W^{u}(x,R_{i})^{o}.
	$$
\end{Definition}

Given a Markov section $\{R_{i}\}_{i=1}^{N}$, let $r:R^{o} \cap  P^{-1}R^{o}\to\mathbb{R}^{+}$ be the first return time function for $P: R^{o} \cap  P^{-1}R^{o} \mapsto R^{o}$. 
We call the set $R_{i}^{r}:=
\overline{\cup_{x\in R^{o} \cap  P^{-1}R^{o}}\cup_{0\le t< r(x)}g_{t}(x)}$ a parallelepiped, and we call $\{R_{i}^{r}\}_{i=1}^{N}$ a Markov partition. The maximal value of the sizes of $R_{i}^{r}$, $1\le i\le N$, is called the size of $\{R_{i}^{r}\}_{i=1}^{N}$. The following well-known result is attributed to Bowen \cite{Bow73}.

\begin{Lemma}\label{Markov partition of hyperbolic flows}
	There exists a Markov partition $\{R_{i}^{r}\}_{i=1}^{N}$ of $g_{t}:\Lambda\to\Lambda$ of arbitrarily small size.
\end{Lemma}

We can naturally view $\Lambda$ as the suspension space of $R$ under the function $r$ and view $g_{t}$ is the corresponding suspension flow of $P:R\to R$ and $r$. We can directly study the dynamics of $P$ and the suspension flow. But, it may be more convenient to view $R$ as a two-sided symbolic space and view $P:R\to R$ as a two-sided subshift. To this end, we can define a transition $N\times N$ matrix $A$ of zeros and ones by
$$
A(i,j)=\begin{cases}
	1,\quad & \text{if}\ R^{o}_{i}\cap P^{-1}R_{j}^{o}\not=0;\\
	0,\quad &\text{otherwise.}
\end{cases} 
$$
We can always assume $A$ is aperiodic, that is $A^{N}>0$ for some $N\in\mathbb{N}^{+}$. Let $\sigma:X\to X$ be the two-sided subshift. By \cite{Bow73}, we have the following result.

\begin{Lemma}\label{Semi-conjugacy of poincare map and subshifts}
	There exists a Lipschitz continuous surjection $\Pi:X\to R$ such that $P:R\to R$ and $\sigma:X\to X$ are semi-conjugated ,i.e. $\Pi\circ\sigma=P\circ\Pi$.
\end{Lemma}

By considering $r\circ\Pi$, we can think of $r$ is a roof function on $X$, and $r$ only depends on future coordinates, i.e. $r(x)=r(y)$ whenever $x_{i}=y_{i}$ for any $i\ge0$. By choosing the metric constant $\lambda\in(0,1)$ on $X$ close sufficiently to 1, we can assume $r$ belongs to $F_{\lambda}(X,\mathbb{R})$. Let $\phi_{t}:X_{r}\to X_{r}$ be the suspension flow of $\sigma$ under $r$ introduced in Subsection \ref{subsec 2.2}. The suspension space $X_{r}$ is a symbolic coordinate system of $\Lambda$ via $\Pi_{r}:X_{r}\to\Lambda$ defined by $\Pi_{r}(x,u)=g_{u}(\Pi(x))$, and $\phi_{t}$ is semi-conjugated to $g_{t}$ by $\Pi_{r}$.

\subsection{Deducing the Diophantine condition}\label{subsec 3.2}

From now on, we assume hyperbolic flow $g_{t}$ satisfies the Diophntine assumption in Theorem \ref{Theorem 1.2}. In particular, this ensures the following conclusion on the suspension flow $\phi_{t}$ which will be used in the proof of the Dolgopyat type estimate in Proposition \ref{Dolgopyat type estimate}.

\begin{Lemma}\label{Lemma 3.5}
There exist $C_{2}>0$, $C_{3}>0$, $C_{4}>0$, $w_{0}\in X$ and $N_{0}\in\mathbb{N}^{+}$ such that for any $|b|\ge C_{4}$ there exist $w_{1}, w_{2}, w_{3}\in X$ with $w_{1}\in W^{u}_{N_{0}}(w_{0})$, $w_{2}\in W^{s}_{N_{0}}(w_{0})$ and $w_{3}\in W^{u}_{N_{0}}(w_{2})\cap W^{s}_{N_{0}}(w_{1})$ such that $$\big|e^{ib(\Delta^{u}(w_{0},w_{1})+\Delta^{s}(w_{1},w_{3})+\Delta^{u}(w_{3},w_{2})+\Delta^{s}(w_{2},w_{0}))}-1\big|\ge C_{3} |b|^{-C_{2}}.$$
\end{Lemma}
\begin{proof}
Fix a point $z_{0}\in\Lambda$. By the Diophantine assumption of Theorem \ref{Theorem 1.2}, there exist $\alpha, \beta\in \text{range}(\mathcal{T})$, $C_{5}>0$ and $C_{6}>0$ such that $|q\frac{\alpha}{\beta}-p|\ge C_{6}|q|^{-C_{5}}$ for any $p\in\mathbb{Z}$ and any $0\not=q\in\mathbb{Z}$. We claim that for any $|b|\ge C_{4}$ we have $|e^{ib\alpha}-1|\ge 2^{-1}C_{6}|b|^{-C_{5}}$ or $|e^{ib\beta}-1|\ge 2^{-1}C_{6}|b|^{-C_{5}}$. Indeed, if not we then obtain that for some $|b|\ge C_{4}$, 
\begin{equation}\label{3.1}
|e^{ib\alpha}-1|\le 2^{-1}C_{6}|b|^{-C_{5}}\quad\text{and}\quad|e^{ib\beta}-1|\le 2^{-1}C_{6}|b|^{-C_{5}}.
\end{equation}
We can choose $n_{\alpha}, n_{\beta}\in\mathbb{Z}$ such that
$$d(b\alpha,2\pi\mathbb{Z})=\min_{k\in\mathbb{Z}}\{|b\alpha-2\pi k|\}=b\alpha-2\pi n_{\alpha}\quad\text{and}\quad d(b\beta,2\pi\mathbb{Z})=\min_{k\in\mathbb{Z}}\{|b\beta
-2\pi k|\}=b\beta-2\pi n_{\beta}.$$
Since $\beta\not=0$, we can choose $C_{4}>0$ to be large sufficiently such that $|b\beta|\ge 2\pi$ and therefore $n_{\beta}\not=0$. By \eqref{3.1}, we have 
\begin{equation*}
	|b\alpha-2\pi n_{\alpha}|\le 2^{-1}C_{6}|b|^{-C_{5}}\quad\text{and}\quad|b\beta-2\pi n_{\beta}|\le 2^{-1}C_{6}|b|^{-C_{5}}.
\end{equation*}
In particular, the above implies $|n_{\beta}\frac{\alpha}{\beta}-n_{\alpha}|\le C_{6}|b|^{-C_{5}}$. Note that $|n_{\beta}|\le \frac{\beta+1}{2\pi}|b|$. Thus, $|n_{\beta}\frac{\alpha}{\beta}-n_{\alpha}|\le C_{6}|n_{\beta}|^{-C_{5}}$ which contradicts the fact that $\frac{\alpha}{\beta}$ is a Diophantine number.

Now, without loss of generality, we assume $|e^{ib\alpha}-1|\ge 2^{-1}C_{6}|b|^{-C_{5}}$. By definition, there exist $z_{1}\in W^{u}_{\varepsilon_{0}}(z_{0})$, $z_{2}\in W^{s}_{\varepsilon_{0}}(z_{0})$, $z_{3}\in W^{u}_{\varepsilon_{0}}(z_{2})$ and $z_{4}\in W^{s}_{\varepsilon_{0}}(z_{1})$ such that $g_{\mathcal{T}(z_{1},z_{2})}(z_{3})=z_{4}$ and $\mathcal{T}(z_{1},z_{2})=\alpha$, where $\mathcal{T}(z_{1},z_{2})$ is the temporal distance of $z_{1}$ and $z_{2}$. Assume $z_{0}, z_{1}, z_{2} $ and $z_{3}$ have symbolic coordinates $(w_{0},u_{0}), (w_{1},u_{1}), (w_{2},u_{2})$ and $ (w_{3},u_{3})\in X_{r}$ respectively. Then, $z_{4}$ has symbolic coordinate $(w_{3},u_{3}+\alpha)$ and we have the following identities:
$$u_{1}=u_{0}+\Delta^{u}(w_{0},w_{1}),\ u_{2}=u_{0}+\Delta^{s}(w_{0},w_{2}),\ u_{3}=u_{2}+\Delta^{u}(w_{2},w_{3})\text{ and } u_{3}+\alpha=u_{1}+\Delta^{s}(w_{1},w_{3}).$$
The above implies that $\Delta^{s}(w_{0},w_{2})+\Delta^{u}(w_{2},w_{3})+\alpha=\Delta^{u}(w_{0},w_{1})+\Delta^{s}(w_{1},w_{3})$. In particular, $\alpha=\Delta^{u}(w_{0},w_{1})+\Delta^{s}(w_{1},w_{3})-\Delta^{u}(w_{2},w_{3})-\Delta^{s}(w_{0},w_{2})=\Delta^{u}(w_{0},w_{1})+\Delta^{s}(w_{1},w_{3})+\Delta^{u}(w_{3},w_{2})+\Delta^{s}(w_{2},w_{0})$ which completes the proof.
\end{proof}

\subsection{Transfer operators and a Dolgopyat type estimate}\label{subsec 3.3}

The potential $\Phi$ on $\Lambda$ induces a potential \(\varphi\in F_{\lambda}(X,\mathbb{R})\). By adding a coboundary \(h\circ\sigma-h\) to $\varphi$ if necessary, we can also assume \(\varphi\) depends only on future coordinates and thus \(\varphi\in F_{\lambda}(X^{+},\mathbb{R})\).  Then, it is well-known \cite{Bow08} that there exists a unique equilibrium state of \(\varphi\) on \(X^{+}\), which is called the Gibbs measure of \(\varphi\) on \(X^{+}\), denoted as \(\mu\). For each \(n\in\mathbb{N}^{+}\), let \(\varphi_{n}=\sum_{i=0}^{n-1}\varphi\circ\sigma^{i}\). The Gibbs measure \(\mu\) satisfies the following Gibbs property, see also \cite{Bow75}.

\begin{Lemma}\label{Gibbs property}
	There exists $C_{7}\geq 1$ such that 
	$$C_{7}^{-1}\le \dfrac{\mu[x_{0}\cdots x_{n-1}]_{n}}{e^{\varphi_{n}(x)-nP(\varphi)}}\le C_{7},$$
	for any $n\in\mathbb{N}^{+}$ and any $x\in X^{+}$ where $P(\varphi)$ is the pressure of $\varphi$.
\end{Lemma}

It is classical \cite{Bow08} that $\mu$ is the eigenmeasure of the transfer operator of $\varphi$, which is defined as follows:
$$\mathcal{L}_{\varphi}:F_{\lambda}(X^{+})\to F_{\lambda}(X^{+}),\quad\mathcal{L}_{\varphi}h(x)=\sum_{\sigma(y)=x}e^{\varphi(y)}h(y).$$
We can further assume $\varphi$ is normalized ,i.e., $\mathcal{L}_{\varphi}1=1$. In particular, we have $P(\varphi)=0$ and $\mathcal{L}_{\varphi}^{*}\mu=\mu$. By a suitable choice of the Markov partition, we can also assume the roof function \(r\) depends only on future coordinates and thus \(r\in F_{\lambda}(X^{+},\mathbb{R})\). To show $g_{t}$ is rapidly mixing, as explained in \cite{Dol98b}, it is sufficient to study the following complex transfer operator:
$$\mathcal{L}_{\varphi+ibr}:F_{\lambda}(X^{+})\to F_{\lambda}(X^{+}),\quad\mathcal{L}_{\varphi+ibr}h(x)=\sum_{\sigma(y)=x}e^{(\varphi+ibr)(y)}h(y),$$
where $b\in\mathbb{R}$. For convenience, denote by $\mathcal{L}_{b}:=\mathcal{L}_{\varphi+ibr}$. In the next section, we will use Lemma \ref{Lemma 3.5} to prove the following Dolgopyat type estimate.

\begin{Proposition}\label{Dolgopyat type estimate}
There exist $C_{8} > 0$, $C_{9}>0$ and $C_{10} > 0$ such that for any $|b|\ge C_{4}$ and any $h \in F_{\lambda}(X^{+})$, we have $||\mathcal{L}_{b}^{C_{8}\log|b|}h||_{b} \le (1 - |b|^{-C_{9}}) ||h||_{b}$, where $||h||_{b}=\max\{|h|_{\infty},|h|_{Lip}/C_{10}|b|\}$.
\end{Proposition}

\begin{proof}[\textbf{Proof of Theorems \ref{Theorem 1.2} and \ref{Theorem 1.3}}]
The estimate of $\mathcal{L}_b$ in Proposition \ref{Dolgopyat type estimate} also provides a similar estimate for the following transfer operator:
$$\mathcal{L}_{\varphi+sr}: F_{\lambda}(X^+) \to F_{\lambda}(X^+), \quad \mathcal{L}_{\varphi+sr}h(x) = \sum_{\sigma(y)=x} e^{(\varphi+sr)(y)} h(y),$$
for $s = a + ib$ with $|a| \le |b|^{-C_9}$ and $|b| \ge C_4$. Then, $g_t$ is rapidly mixing with respect to \(\mu_{\Phi}\) follows from the argument in \cite{Dol98b}, where all missing details can be found in \cite{Mel18}. Thus, Theorem \ref{Theorem 1.2} is proved. Additionally, following the proof in \cite{Pol01}, the same estimate on \(\mathcal{L}_{\varphi+sr}\) ensures a polynomial error term in the Prime Orbit Theorem, which proves Theorem~\ref{Theorem 1.3}.
\end{proof}

\section{Proof of Proposition \ref{Dolgopyat type estimate}}\label{sec 4}

We follow the strategy in \cite[\S 7.3]{Pol24} to prove Proposition \ref{Dolgopyat type estimate}. The main ingredient in the proof is a cancellation estimate stated in Lemma \ref{Lemma 4.4} which will be proved in the next section.

\begin{Lemma}\label{Lemma 4.1}
There exists $C_{11}>0$ such that for any $|b|\ge C_{4}$, any $h\in F_{\lambda}(X^{+})$ and any $n\ge\mathbb{N}^{+}$,
	$$|\mathcal{L}_{b}^{n}h|_{Lip}\le C_{11}|b||h|_{\infty}+\lambda^{n}|h|_{Lip}.$$
\end{Lemma}
\begin{proof}
The proof can be obtained by direct calculation or by referring to \cite{Dol98a}.
\end{proof}

Recalling, for any $|b|\ge C_{4}$, we introduce a norm on $F_{\lambda}(X^{+})$ by \(||h||_{b}=\max\{|h|_{\infty},|h|_{\text{Lip}}/C_{10}|b|\}\). Fix a $\lambda^{\prime}\in(\lambda,1)$, and then we choose \(C_{10}\geq 1\) such that \(C_{11}/C_{10}+\lambda\le\lambda^{\prime}\). 

\begin{Corollary}\label{Cor 4.2}
For any $|b|\ge C_{4}$, any \(n\ge\mathbb{N}^{+}\), and any \(h\in F_{\lambda}(X^{+})\), we have \(\frac{|\mathcal{L}^{n}_{b}h|_{\text{Lip}}}{C_{10}|b|}\le\lambda^{\prime}||h||_{b}\) and \(||\mathcal{L}^{n}_{b}h||_{b}\le||h||_{b}\).
\end{Corollary}
\begin{proof}
By Lemma \ref{Lemma 4.1}, for any  $|b|\ge C_{4}$  , any \(n\ge\mathbb{N}^{+}\), and any \(h\in F_{\lambda}(X^{+})\),
\begin{equation*}
		\dfrac{|\mathcal{L}^{n}_{b}h|_{\text{Lip}}}{C_{10}|b|}\le \dfrac{C_{11}}{C_{10}}|h|_{\infty}+\dfrac{\lambda^{n}}{C_{10}|b|}|h|_{\text{Lip}}.
	\end{equation*}
	In particular,
	\begin{equation*}
		\dfrac{|\mathcal{L}^{n}_{b}h|_{\text{Lip}}}{C_{10}|b|}\le\bigg(\dfrac{C_{11}}{C_{10}}+\lambda^{n}\bigg)||h||_{b}.
	\end{equation*}
	Note that \(|\mathcal{L}^{n}_{b}h|_{\infty}\le|h|_{\infty}\). Thus, provided \(\frac{C_{11}}{C_{10}}+\lambda\le\lambda^{\prime}\), we have \(||\mathcal{L}^{n}_{b}h||_{b}\le||h||_{b}\) which completes the proof.
\end{proof}

\begin{Lemma}\label{Lemma 4.3}
For any   $|b|\ge C_{4}$, any $h\in F_{\lambda}(X^{+})$ with $|h|_{Lip}\ge 2C_{10}|b||h|_{\infty}$ and any $n\ge\mathbb{N}^{+}$, we have $||\mathcal{L}^{n}_{b}h||_{b}\le\lambda^{\prime}||h||_{b}$.
\end{Lemma}
\begin{proof}
Since $\lambda>0$ is close to 1, we can assume $\lambda\ge 1/2$. We have $|\mathcal{L}^{n}_{b}h|_{\infty}\le|h|_{\infty}\le|h|_{Lip}/2C_{10}b_{m}\le2^{-1}||h||_{b}\le\lambda||h||_{b}$. Thus, by Corollary \ref{Cor 4.2}, we have $||\mathcal{L}^{n}_{b}h||_{b}\le\lambda^{\prime}||h||_{b}$.
\end{proof}

In the next section, we will use the conclusion obtained in Lemma \ref{Lemma 3.5} to prove the following cancellation of terms in a transfer operator.

\begin{Lemma}\label{Lemma 4.4}
There exist $C_{12}>0, C_{13}>0$ and $C_{14}>0$ such that for any $|b|\ge C_{4}$ and any $h\in F_{\lambda}(X^{+})$ with $|h|_{Lip}\le 2C_{10}|b||h|_{\infty}$, there exists a subset $U\subset X^{+}$ with $\mu(U)\ge |b|^{-C_{12}}$ such that $|\mathcal{L}^{C_{13}\log |b|}_{b}h(x)|\le(1-|b|^{-C_{14}})|h|_{\infty}$ for any $x\in U$.
\end{Lemma}

As a corollary of Lemma \ref{Lemma 4.4}, we can obtain a cancellation of the oscillatory integral $\int_{X^{+}}|\mathcal{L}_{b}^{C_{13}\log |b|}h|d\mu$ as follows:
\begin{equation}\label{4.1}
	\begin{aligned}
		\int_{X^{+}}|\mathcal{L}_{b}^{C_{13}\log |b|}h|d\mu=&\int_{U}|\mathcal{L}_{b}^{C_{13}\log|b|}h|d\mu+\int_{X^{+}-U}|\mathcal{L}_{b}^{C_{13}\log|b|}h|d\mu\\
		\le&(1-|b|^{-C_{14}})|h|_{\infty}\mu(U)+|h|_{\infty}\mu(X^{+}-U)\\
		=&(1-|b|^{-C_{14}}\mu(U))|h|_{\infty}\le(1-|b|^{-C_{14}-C_{12}})|h|_{\infty}.
	\end{aligned}
\end{equation}
To strengthen the above $L^{1}$ contraction to a $|\cdot|_{\infty}$-contraction, we require the following lemma. 

\begin{Lemma}\label{Lemma 4.5}
	There exist $C_{15}>0$ and $\delta\in(0,1)$ such that for any $h\in F_{\lambda}(X^{+})$ and any $k\ge\mathbb{N}^{+}$,
	$$||\mathcal{L}_{\varphi}^{k}h||_{Lip}\le \int_{X^{+}}|h|d\mu+C_{15}\delta^{k}||h||_{Lip}.$$
\end{Lemma}
\begin{proof}
	This is a directly corollary of the spectral gap of $\mathcal{L}_{\varphi}$ acts on $F_{\lambda}(X^{+})$ \cite{Bal00}.
\end{proof}

\begin{Lemma}\label{Lemma 4.6}
	There exist $C_{16}>0$ and $ C_{17}>0$ such that for any $|b|\ge C_{4}$ and any $h\in F_{\lambda}(X^{+})$ with $|h|_{Lip}\le 2C_{10}|b||h|_{\infty}$,
	$$|\mathcal{L}_{b}^{(C_{17}+C_{13})\log |b|}h|_{\infty}\le\bigg(1-|b|^{-C_{16}}\bigg)|h|_{\infty}.$$
\end{Lemma}
\begin{proof}
	By \eqref{4.1}, Corollary \ref{Cor 4.2} and Lemma \ref{Lemma 4.5}, we have
	$$
	\begin{aligned}
		&|\mathcal{L}_{b}^{(C_{17}+C_{13})\log |b|}h|_{\infty}\le|\mathcal{L}_{\varphi}^{C_{17}\log |b|}|\mathcal{L}_{b}^{C_{13}\log |b|}h||_{\infty}\\
		\le&\int_{X^{+}}|\mathcal{L}_{b}^{C_{13}\log|b|}h|d\mu+C_{15}\delta^{C_{17}\log |b|}||\mathcal{L}_{b}^{C_{13}\log|b|}h||_{Lip}\\
		\le&(1-|b|^{-C_{14}-C_{12}})|h|_{\infty}+2C_{15}\delta^{C_{17}\log |b|}C_{10}|b||h|_{\infty}\\
		\le&(1-|b|^{-C_{16}})|h|_{\infty}
	\end{aligned}
	$$
	which completes the proof.
\end{proof}

The above $|\cdot|_{\infty}$-contraction implies the following $||\cdot||_{b_{m}}$-contraction.

\begin{Corollary}\label{Cor 4.7}
	For any $|b|\ge C_{4}$ and any $h\in F_{\lambda}(X^{+})$ with $|h|_{Lip}\le 2C_{10}|b||h|_{\infty}$,
	$$||\mathcal{L}_{b}^{(C_{17}+C_{813})\log |b|}h||_{b}\le\bigg(1-|b|^{-C_{16}}\bigg)||h||_{b}.$$
\end{Corollary}
\begin{proof}
	This comes from Corollary \ref{Cor 4.2} and Lemma \ref{Lemma 4.6}.
\end{proof}

\begin{proof}[\textbf{Proof of Proposition \ref{Dolgopyat type estimate}}]
	This comes from Lemma \ref{Lemma 4.3} and Corollary \ref{Cor 4.7}.
\end{proof}

\section{Proof of Lemma \ref{Lemma 4.4}}\label{sec 5}

The proof of the cancellation lemma: Lemma \ref{Lemma 4.4} will use the Diophantine property obtained in Lemma \ref{Lemma 3.5} which follows from the Diophantine assumption on the temporal distance function. Previously, such kind of cancellation lemmas can be obtained by using the assumption that the range of the temporal distance function has positive lower box dimension \cite{Dol98b}.

We begin with some simplifying observations. The first point to observe  is that achieving cancellation at a single point is sufficient. Indeed, according to Lemma \ref{Lemma 4.1}, the estimate $|h|_{\text{Lip}} \le 2C_{10}|b| |h|_{\infty}$ implies $|\mathcal{L}_{b}^{C_{13}\log|b|}h|_{\text{Lip}} \le 2C_{10}|b| |h|_{\infty}$. Therefore, if $|\mathcal{L}_{b}^{C_{13}\log|b|}h(x)| \le (1-|b|^{-C_{14}})|h|_{\infty}$, then for any $y \in X^{+}$ with $d_{\lambda}(x,y) \le (4C_{10}|b|^{C_{14}+1})^{-1}$, we have
\[
\begin{aligned}
|\mathcal{L}_{b}^{C_{13}\log|b|}h(y)| &\le |\mathcal{L}_{b}^{C_{13}\log|b|}h(y) - \mathcal{L}_{b}^{C_{13}\log |b|}h(x)| + (1-|b|^{-C_{14}})|h|_{\infty}\\
	&\le  \dfrac{1}{2}|b|^{-C_{14}}|h|_{\infty} + (1-|b|^{-C_{14}})|h|_{\infty}= (1-2^{-1}|b|^{-C_{14}})|h|_{\infty}.
\end{aligned}
\]
Using the Gibbs property of $\mu$ (in Lemma \ref{Gibbs property}),  it is  easy to show that the set of these points $y$ has $\mu$-measure $\ge |b|^{-C_{12}}$ for some uniform constant $C_{12} > 0$. Then  Lemma \ref{Lemma 4.4}  would follow.

The second point we want to note is that we can always assume $h$ satisfies $|h(x)| \ge |h|_{\infty}3/4$ for any $x \in X^{+}$. Indeed, if not, then for some point $x \in X^{+}$, we would have $|h(x)| \le 3|h|_{\infty}/4$. Consequently, we would obtain
\[
\begin{aligned}
	&|\mathcal{L}_{b}^{C_{13}\log |b|}h(\sigma^{C_{13}\log |b|}(x))|\\
	\le&\sum_{\substack{\sigma^{C_{13}\log |b|}(y)=\sigma^{C_{13}\log |b|}(x)\\y\not=x}}e^{\varphi_{C_{13}\log|b|}(y)}|h|_{\infty} + e^{\varphi_{C_{13}\log |b|}(x)}3|h|_{\infty}/4\\
	\le&(1-e^{-C_{13}\log|b||\varphi|_{\infty}}/4)|h|_{\infty} \le (1-|b|^{-C_{14}})|h|_{\infty}.
\end{aligned}
\]
Then the required result would follow  from the argument presented in the previous paragraph.

We will  use the following elementary  inequality to cancel the terms in the  transfer operator.

\begin{Lemma}\label{Lemma 5.1}
Assume $0\not=v_{1},v_{2}\in\mathbb{C}$. If $|\frac{v_{1}}{|v_{1}|}-\frac{v_{2}}{|v_{2}|}||\ge\varepsilon$ and $|v_{1}|\le|v_{2}|$, then $|v_{1}+v_{2}|\le (1-\varepsilon^{2}/4)|v_{1}|+|v_{2}|$. 
\end{Lemma}

Denote $n_{b} = C_{13}\log|b|$ and $\eta = h/|h|$. Note that $|\eta| \equiv 1$. Given that $|h|_{\text{Lip}} \le 2C_{10}|b||h|_{\infty}$ and $|h(x)| \ge 3|h|_{\infty}/4$ for any $x \in X^{+}$, it is straightforward to show that $|\eta|_{\text{Lip}} \le 2C_{10}|b|$. Our objective is to demonstrate the existence of $x \in X^{+}$ and $y_{1}, y_{2} \in \sigma^{-n_{b}}(x)$ such that
\begin{equation}\label{5.1}
	|e^{ibr_{n_{b}}(y_{1})}\eta(y_{1}) - e^{ibr_{n_{b}}(y_{2})}\eta(y_{2})| \ge |b|^{-C_{18}},
\end{equation}
for some uniform constant $C_{18} > 0$. 
Then, by Lemma \ref{Lemma 5.1},
\[
\begin{aligned}
	&|\mathcal{L}_{b}^{n_{b}}h(x)|\\
	\le & \sum_{\sigma^{n_{b}}(y)=x; y \neq y_{1},y_{2}}e^{\varphi_{n_{b}}(y)}|h(y)|+ \big|e^{\varphi_{n_{b}}(y_{1})}e^{ibr_{n_{b}}(y_{1})}h(y_{1})+e^{\varphi_{n_{b}}(y_{2})}e^{ibr_{n_{b}}(y_{2})}h(y_{2})\big| \\
	\le & \sum_{\sigma^{n_{b}}(y)=x; y \neq y_{1},y_{2}}e^{\varphi_{n_{b}}(y)}|h(y)| + (1-|b|^{-2C_{18}}/4)e^{\varphi_{n_{b}}(y_{1})}|h(y_{1})|+e^{\varphi_{n_{b}}(y_{2})}|h(y_{2})| \\
	\le & (1 - e^{-n_{b}|\varphi|_{\infty}}|b|^{-2C_{18}})|h|_{\infty} \le (1-|b|^{-C_{14}})|h|_{\infty},
\end{aligned}
\]
for some uniform constant $C_{14} > 0$. Then Lemma \ref{Lemma 4.4} would follow from the argument presented at the beginning of this section.

Recalling the constants in Lemma \ref{Lemma 3.5}. For any $|b|\ge C_{4}$, by Lemma \ref{Lemma 3.5}, there exist $w_{1}, w_{2}, w_{3}\in X$ with $w_{1}\in W^{u}_{N_{0}}(w_{0})$, $w_{2}\in W^{s}_{N_{0}}(w_{0})$ and $w_{3}\in W^{u}_{N_{0}}(w_{2})\cap W^{s}_{N_{0}}(w_{1})$ such that 
\begin{equation}\label{5.2}
\big|e^{-ib(\Delta^{u}(w_{0},w_{1})+\Delta^{s}(w_{1},w_{3})+\Delta^{u}(w_{3},w_{2})+\Delta^{s}(w_{2},w_{0}))}-1\big|\ge C_{3} |b|^{-C_{2}}.
\end{equation}
Denote by $\Pi_{+}:X\to X^{+}$ the coordinate projection. For each $0 \le j \le 3$, we define 
$$x_{j} := \Pi_{+}\sigma^{-(n_{b}-N_{0})}(w_{j}) \in X^{+}\quad\text{and}\quad I_{j}=e^{ibr_{n_{b}-N_{0}}(x_{j})}\eta(x_{j}).$$
Note that $|I_{j}|=1$ for each $0\le j \le 3$. We can express
\begin{equation}\label{5.3}
\begin{aligned}
	& I_{0} -e^{-ib(\Delta^{u}(w_{0},w_{1})+\Delta^{s}(w_{1},w_{3})+\Delta^{u}(w_{3},w_{2})+\Delta^{s}(w_{2},w_{0}))}I_{0}\\
	=&\big(I_{1} - e^{-ib\Delta^{u}(w_{0},w_{1})}I_{0}\big)e^{-ib(\Delta^{s}(w_{1},w_{3})+\Delta^{u}(w_{3},w_{2})+\Delta^{s}(w_{2},w_{0}))}\\
	&+\big(I_{3} - e^{-ib\Delta^{s}(w_{1},w_{3})}I_{1}\big)e^{-ib(\Delta^{u}(w_{3},w_{2})+\Delta^{s}(w_{2},w_{0}))}\\
	&+\big(I_{2} - e^{-ib\Delta^{u}(w_{3},w_{2})}I_{3}\big)e^{-ib(\Delta^{s}(w_{2},w_{0}))}\\
	&+\big(I_{0} - e^{-ib\Delta^{s}(w_{2},w_{0})}I_{2}\big).
\end{aligned}
\end{equation}

\begin{Lemma}\label{Lemma 5.2}
We have $\big|I_{1} - e^{-ib\Delta^{u}(w_{0},w_{1})}I_{0}\big|\le \frac{1}{4}C_{3}|b|^{-C_{2}}$ and $\big|I_{2} - e^{-ib\Delta^{u}(w_{3},w_{2})}I_{3}\big|\le\frac{1}{4}C_{3}|b|^{-C_{2}}$.
\end{Lemma}
\begin{proof}
Since $w_{1}\in W^{u}_{N_{0}}(w_{0})$, we have $(x_{1})_{i}=(x_{0})_{i}$ for any $i\le n_{b}-2N_{0}$. Since $|\eta|_{\text{Lip}}\le 2C_{10}|b|$, we have
\begin{equation}\label{5.4}
|\eta(x_{1})-\eta(x_{0})|\le 2C_{10}|b|\lambda^{n_{b}-2N_{0}}\le \dfrac{1}{8}C_{3}|b|^{-C_{2}},
\end{equation}
provided $C_{13}>0$ is large enough where $n_{b}=C_{13}\log |b|$. Meanwhile, by Lemma \ref{Lemma 2.4}, 
	\begin{equation}\label{5.5}
		\begin{aligned}
			&|\Delta^{u}(w_{0},w_{1})-[r_{n_{b}-N_{0}}(x_{0})-r_{n_{b}-N_{0}}(x_{1})]|\\
			=&|\Delta^{u}(w_{0},w_{1})-[r_{n_{b}-N_{0}}(\sigma^{-n_{b}+N_{0}}(w_{0}))-r_{n_{b}-N_{0}}(\sigma^{-n_{b}+N_{0}}(w_{1}))]|\\
			\le &C_{1}\lambda^{n_{b}-N_{0}}\le\dfrac{1}{8}C_{3}|b|^{-C_{2}-1},
		\end{aligned}
	\end{equation}
	provided $C_{13}>0$ is large enough. Now, we can relate the bound \eqref{5.5}  to the estimate below.
	$$
	\begin{aligned}
		&|\eta(x_{1})-\eta(x_{0})|\\
		=&|e^{-ibr_{n_{b}-N_{0}}(x_{1})}I_{1}-e^{-ibr_{n_{b}-N_{0}}(x_{0})}I_{0}|\\
		=&|I_{1}-e^{-ib[r_{n_{b}-N_{0}}(x_{0})-r_{n_{b}-N_{0}}(x_{1})]}I_{0}|\\
		\ge&|I_{1}-e^{-ib\Delta^{u}(w_{0},w_{1})}I_{0}|-|b||\Delta^{u}(w_{0},w_{1})-[r_{n_{b}-N_{0}}(x_{0})-r_{n_{b}-N_{0}}(x_{1})]|\\
		\ge&|I_{1}-e^{-ib\Delta^{u}(w_{0},w_{1})}I_{0}|-\dfrac{1}{8}C_{3}|b|^{-C_{2}}.
	\end{aligned}
	$$
	Together with \eqref{5.4}, we have
	$$|I_{1}-e^{-ib\Delta^{u}(w_{0},w_{1})}I_{0}|\le\dfrac{1}{4}C_{3}|b|^{-C_{2}}$$
	which proves the first bound. The second bound follows the same argument and thus completing the proof.
\end{proof}

Now, using the bounds obtained in Lemma \ref{Lemma 5.2}, together with \eqref{5.2} and \eqref{5.3}, we have
\begin{equation*}
\big|\big(I_{3} - e^{-ib\Delta^{s}(w_{1},w_{3})}I_{1}\big)e^{-ib(\Delta^{u}(w_{3},w_{2})+\Delta^{s}(w_{2},w_{0}))}+\big(I_{0} - e^{-ib\Delta^{s}(w_{2},w_{0})}I_{2}\big)\big|\ge \dfrac{1}{2}C_{3}|b|^{-C_{2}}.
\end{equation*}
In particular, $\big|I_{3} - e^{-ib\Delta^{s}(w_{1},w_{3})}I_{1}\big|\ge \frac{1}{4}C_{3}|b|^{-C_{2}}$ or $\big|I_{0}-e^{-ib\Delta^{s}(w_{2},w_{0})}I_{2}\big|\ge \frac{1}{4}C_{3}|b|^{-C_{2}}$. Without loss of generality, assume $\big|I_{3} - e^{-ib\Delta^{s}(w_{1},w_{3})}I_{1}\big|\ge \frac{1}{4}C_{3}|b|^{-C_{2}}$.
Since $r$ depends only on future coordinates and $w_{3}\in W^{s}_{N_{0}}(w_{1})$, by the definition of $\Delta^{s}$, we have $$\Delta^{s}(w_{1},w_{3})=r_{N_{0}}(\Pi_{+}(w_{3}))-r_{N_{0}}(\Pi_{+}(w_{1})).$$
Therefore,
$$
\begin{aligned}
	\frac{1}{4}C_{3}|b|^{-C_{2}}\le &|I_{3}-e^{-ib\Delta^{s}(w_{1},w_{3})}I_{1}|\\
	=&|I_{3}-e^{-ib[r_{N_{0}}(\Pi_{+}(w_{3}))-r_{N_{0}}(\Pi_{+}(w_{1}))]}I_{1}|\\
	=&|e^{ibr_{n_{b}}(x_{3})}\eta(x_{3})-e^{ibr_{n_{b}}(x_{1})}\eta(x_{1})|,
\end{aligned}
$$
which proves \eqref{5.1} since $\sigma^{n_{b}}(x_{3})=\sigma^{n_{b}}(x_{1})$. Thus the proof of Lemma \ref{Lemma 4.4} is completed.

 \end{document}